\documentclass[11pt]{article}
\usepackage{}
\usepackage{enumerate}
\usepackage{mathbbold}
\usepackage{bbm}
\usepackage{amsfonts}
\usepackage[T1]{fontenc}
\usepackage{latexsym,amssymb,amsmath,amsfonts,amsthm}
\usepackage{graphicx, color}
\usepackage{subfigure}
\usepackage{epsf}
\usepackage{epstopdf}
\usepackage{amssymb}
\usepackage{mathrsfs}

\renewcommand{\theequation}{\arabic{section}.\arabic{equation}}

\newcommand{\R}{{\mathbb R}}

\newcommand{\C}{{\mathbb C}}
\newcommand{\Sp}{{\mathbb S}}
\newcommand{\ds}{\displaystyle}

\newcommand{\be}{\begin{eqnarray}}
\newcommand{\ben}{\begin{eqnarray*}}
\newcommand{\en}{\end{eqnarray}}
\newcommand{\enn}{\end{eqnarray*}}
\newcommand{\ba}{\backslash}
\newcommand{\pa}{\partial}

\newcommand{\ov}{\overline}

\newcommand{\G}{\Gamma}

\newcommand{\om}{\omega}

\newcommand{\wi}{\widetilde}

\newcommand{\hth}{\theta}
\newcommand{\hx}{\hat{x}}
\newtheorem{theorem}{Theorem}[section]
\newtheorem{lemma}[theorem]{Lemma}

\begin{document}
\title{\bf Identification of point like objects with multi-frequency sparse data}
\author{Xia Ji\thanks{LSEC, NCMIS and Academy of Mathematics and Systems Science, Chinese Academy of Sciences,
Beijing 100190, China. Email: jixia@lsec.cc.ac.cn (XJ)}
,\and
Xiaodong Liu\thanks{NCMIS and Academy of Mathematics and Systems Science,
Chinese Academy of Sciences, Beijing 100190, China. Email: xdliu@amt.ac.cn (XL)}}
\date{}
\maketitle

\begin{abstract}
The inverse acoustic scattering of point objects using multi-frequency sparse measurements are studied.
The objects may be a sum of point sources or point like scatterers.
We show that the locations and scattering strengths of the point objects can be uniquely identified by the multi-frequency near or far fields taken at sparse sensors.
Based on the uniqueness analysis, some direct methods have also been proposed for reconstructing the locations and determining the scattering strengths.
The numerical examples are conducted to show the validity and robustness of the proposed numerical methods.

\vspace{.2in}
{\bf Keywords:} inverse scattering; multi-frequency; sparse data; uniqueness; sampling method.

\vspace{.2in} {\bf AMS subject classifications:}
35P25, 45Q05, 78A46, 74B05

\end{abstract}

\section{Introduction}
The inverse scattering theory aims to reconstruct the unknown objects from the wave measurements. This plays an important role in many areas such as radar, nondestructive testing, medical imaging, geophysical prospection and remote sensing. We refer to the standard monograph \cite{CK} for a research statement on the significant progress both in the mathematical theories and the numerical approaches.

A practical difficulty is that the measurements are not easy or even impossible to be taken all around the unknown objects. Actually, the measurements are often available at a few sensors (i.e., a sparse array). At a fixed sensor, it is easy to vary frequency to obtain more data. This is still a small set of data, which indeed brings many difficulties for the solvability of the inverse problems.
The first result is given in 2005 by Sylvester and Kelly \cite{SylvesterKelly}, where they considered the linear inverse acoustic source problems and showed that a convex polygon containing the unknown source  with normals in the observation directions can be uniquely determined by the multi-frequency sparse far field patterns. This implies, even for such a small data set, that one can make a meaningful statement about the size and location of the source.
A factorization method using sparse multi-frequency far field measurements is recently introduced in \cite{GriesmaierSchmiedecke-source} to produce a union of convex polygons that approximate the locations and the geometry of well-separated source components. We also refer to \cite{AlaHuLiuSun}, where a direct sampling method using sparse multi-frequency far field measurements is designed to reconstruct the location and shape. Surprisingly, even the concave part of the source support can be well reconstructed with enough observation directions. However, the corresponding theoretical basis is still not established. The direct sampling method is also generalized for inverse acoustic, elastic and electromagnetic source scattering problems with phased or phaseless multi-frequency sparse far field data \cite{JiLiu-elastic,JiLiu-electromagnetic, JiLiuZhang-source}. It is also shown in \cite{AlaHuLiuSun} that the smallest annular centered at the sensor that containing the source support in $\R^3$ can be uniquely determined by the multi-frequency scattered fields taken at the sensor. Difficulties arise for the inverse obstacle/medium scattering problems because these problems are nonlinear. Sylvester and Kelly \cite{SylvesterKelly} considered Born approximation to the inverse medium problem and obtain similar results with the inverse source problem.
The MUSIC (MUltiple-SIgnal-Classification) algorithm \cite{GriesmaierSchmiedecke} is studied for locating small inhomogeneities.
Based on the weak scattering approximation and the Kirchhoff approximation, a direct sampling method is proposed for location and shape reconstruction of the underlying objects \cite{JiLiu-sparse} by using multi-frequency sparse back-scattering far field measurements.

In practical radar and remote sensing, faraway objects radiate fields that, within measurement precision, are nearly those radiated by point like objects.
The objects may be a sum of point sources or point like scatterers. This paper aims to identify their locations and strengths based
on multi-frequency sparse near or far fields. Our main focus in this paper is the uniqueness theories and numerical algorithms for determining the point like
objects.
In the recent manuscript \cite{JiLiu-sparse}, the point like obstacles have been considered with multi-frequency sparse backscattering far field patterns.
This paper clarifies the smallest number of sensors to be used.
Furthermore, we show the uniqueness of the scattering strengths and introduce the corresponding formula.
Based on the uniqueness analyses, some novel direct methods are designed to locate the points and to reconstruct the scattering strengths.
We refer to \cite{ZGLL} for locating point sources using direct sampling methods with the single frequency far field patterns at all the observation directions.

The remaining part of the work is organized as follows.
In the next section, we introduce the inverse scattering of point sources and the inverse scattering of plane waves by point like scatterers.
We then proceed in the Section \ref{uniqueness} for the uniqueness results with sparse data.
Section \ref{NumMethods} is devoted to some numerical methods for reconstructing the numbers, locations and strengths of the point objects.
The numerical methods are then verified in Section \ref{NumExamples} by extensive examples.

\section{Scattering by point objects}

We begin with the formulations of the acoustic scattering problem. Let $k=\om/c>0$ be the wave number of
a time harmonic wave, where $\om>0$ and $c>0$ denote the frequency and sound
speed, respectively.
In the whole paper, we consider multiple frequencies in a bounded band, i.e.,
\be\label{kassumption}
k\in (k_{-}, k^{+}),
\en
with two positive wave numbers $k_{-}$ and $k^{+}$.
Recall the fundamental solution $\Phi_k(x,y), x,y\in \R^n, x\neq y,$ of the Helmholtz equation, which is given by
\be\label{Phi}
\Phi_k(x,y):=\left\{
         \begin{array}{ll}
         \ds\frac{ik}{4\pi}h^{(1)}_0(k|x-y|)=\frac{e^{ik|x-y|}}{4\pi|x-y|}, & n=3, \\
         \ds\frac{i}{4}H^{(1)}_0(k|x-y|), & n=2.
         \end{array}
         \right.
\en
Here, $h^{(1)}_0$ and $H^{(1)}_0$ are, respectively, spherical Hankel function and Hankel function of
the first kind and order zero.

{\bf Point sources.} We consider an array of $M$ point sources located at $z_1, z_2,\cdots, z_M\in\R^{n}$ in the homogeneous space $\R^{n},n=2,3$, denote by $\tau_m\in\C\ba\{0\}$ the scattering strength of the $m$-th point source.
The scattered field $u^s$ is a solution of the following equation
\ben
\Delta u^s +k^2 u^s = \sum_{m=1}^{M}\tau_m\delta_{z_m} \quad \mbox{in}\,\, \R^n,
\enn
$\delta_{z_m}$ denoting the Dirac measure on $\R^n$ giving unit mass to the point $z_m,\,m=1,2,\cdots,M$.
Precisely, the scattered field $u^s$ is given by
\be\label{us-source}
u^s(x,k)=\sum_{m=1}^{M}\tau_m\Phi_k(x,z_m), \quad x\in\R^n\ba\{z_1, z_2,\cdots, z_M\}.
\en
From the asymptotic behavior of $\Phi_k(x,y)$ we conclude that
\be\label{sourceasyrep}
u^s(x,k)
=\frac{e^{i\frac{\pi}{4}}}{\sqrt{8k\pi}}\left(e^{-i\frac{\pi}{4}}\sqrt{\frac{k}{2\pi}}\right)^{n-2}
\frac{e^{ikr}}{r^{\frac{n-1}{2}}}\left\{\sum_{m=1}^{M}\tau_j e^{-ik\hx\cdot z_j}
+\mathcal{O}\left(\frac{1}{r}\right)\right\}\quad\mbox{as }\,r:=|x|\rightarrow\infty,
\en
Therefore, the far field pattern is given by
\be\label{uinf-source}
u^{\infty}(\hx,k)=\sum_{m=1}^{M}\tau_m e^{-ik\hx\cdot z_m},
\en
where $\hx\in \Sp^{n-1}:=\{x\in\R^n:|x|=1\}$ denotes the observation direction.

{\bf Point like scatterers.} The second case of our interest is the scattering of plane waves by $M$ point like scatterers located at $z_1, z_2,\cdots, z_M\in\R^{n}$ in the homogeneous space $\R^{n},n=2,3$. The incident plane wave $u^{i}$ is of the form
\be\label{incidenwave}
u^{i}(x,\hth,k) = e^{ikx\cdot \theta},\quad x\in\R^n,
\en
where $\theta\in\Sp^{n-1}$ denotes the direction of the incident wave.

By neglecting all the multiple scattering between the scatterers, the scattered field $u^{s}$ is given by \cite{Foldy}
\be\label{uszm}
u^{s}(x,\hth,k)= \sum_{m=1}^{M}\tau_m u^{i}(z_m,\hth,k)\Phi_k(x,z_m),
\en
which solves the equation $\Delta u^s(x,\hth,k) +k^2 u^s(x,\hth,k) = \sum_{m=1}^{M}\tau_m u^{i}(z_m,\hth,k)\delta_{z_m}$ in $\R^n$.
Here, $\tau_m\in\C\ba\{0\}$ is the scattering strength of the $m$-th target, $m=1,2,\cdots, M.$
Similarly, the far field pattern is given by
\be\label{uinf-scatterer}
u^{\infty}(\hx,\hth, k)=\sum_{m=1}^{M}\tau_m e^{-ik(\hx-\hth)\cdot z_m}, \quad \hx,\hth\in \Sp^{n-1}.
\en
Of particular interest is the backscattering case, i.e., $\hx=-\hth$.

The inverse problem is to identify the numbers, locations and scattering strengths of the unknown point object from the multi-frequency scattered fields or far fields at a few sensors. We denote by
\ben
\G:=\{x_1, x_2, \cdots, x_L\} \subset \R^n\ba\{z_1, z_2, \cdots, z_M\}
\enn
and
\be\label{thetaL}
\Theta_{L}:=\{ \hx_1,  \hx_2, \cdots, \hx_L \}\in \Sp^{n-1},
\en
respectively, the collection of the sensors for the scattered fields and the far fields.

\section{Uniqueness}
\label{uniqueness}
In this section, we investigate under what conditions a target is uniquely determined by a knowledge of its scattered fields or far field patterns. We note that by
analyticity both the scattered field and its far field pattern is completely determined for all positive frequencies by only knowing them in some bounded band, as given in \eqref{kassumption}.

We begin with the simplest case with a single point source, i.e., $M=1$. In this case, the scattered field and its far field are given, respectively, by
\be\label{1Musuinf}
u^s(x,k)=\tau_1\Phi_k(x,z_1) \quad\mbox{and}\quad u^\infty(\hx, k)=\tau_1 e^{-ik\hx\cdot z_1}.
\en

\begin{theorem}\label{1k1M}
For a fixed frequency $k>0$, let $M=1$, then we have the following results.
\begin{itemize}
  \item For any single sensor $\hx\in\Sp^{n-1},\,n=2,3$ or $x\in\R^3\ba\{z_1\}$, we have
     \be\label{11}
     \tau_1=u^\infty(\hx, k)e^{ik\hx\cdot z_1}, \qquad |x-z_1|=\frac{|\tau_1|}{4\pi|u^s(x,k)|}.
     \en
  \item If we know the location $z_1$ in advance, then the strength $\tau_1$ is uniquely determined by the scattered field $u^s(x,k)$ at a single sensor $ x\in\R^n\ba\{z_1\}$ or the far field pattern $u^{\infty}(\hx, k)$ at a single observation direction $\hx\in\Sp^{n-1},\,n=2,3$.
  \item If we know the scattering strength $\tau_1$ in advance, under the condition that $|k\hx\cdot z_1|<\pi$, then the value $\hx\cdot z_1$ is uniquely determined by the far field pattern $u^{\infty}(\hx, k)$ at a single observation direction $\hx\in\Sp^{n-1},\,n=2,3$. Furthermore, the location $z_1\in\R^n$ can be uniquely determined by $n$ linearly independent observation directions.
  \item If we know the modulus $|\tau_1|$ in advance, then the distance $|x-z_1|$ is uniquely determined by the phaseless scattered field $|u^s(x,k)|$ at a fixed sensor $x\in\R^3$. Furthermore, the location $z_1\in \R^3$ can be uniquely determined by four sensors $x\in\{x_1, x_2, x_3, x_4\}\subset\R^3\ba\{z_1\}$, which are not coplanar.
\end{itemize}
\end{theorem}
\begin{proof}
The first three results are obvious from the representation \eqref{1Musuinf} of the scattered field and its far field pattern. If the modulus $|\tau_1|$ is given in advance, this implies that the location
\ben
z_1\in \pa B_{r_j}(x_j),
\enn
where $\pa B_{r_j}(x_j)$ is a sphere centered at the sensor $x_j$ with radius $r_j:=\frac{|\tau_1|}{4\pi|u^s(x_j,k)|}$, $j=1,2,3,4$.
We give a constructive proof for the determination of the location $z_1$.
With the first two sensors $x_1$ and $x_2$, we
obtain that $z_1$ is located on the circle $\pa B_{r_1}(x_1) \cap \pa B_{r_2}(x_2)$, which is the intersection of two spheres $\pa B_{r_1}(x_1)$ and $\pa B_{r_2}(x_2)$. Since the four sensors $x_1, x_2, x_3$ and $x_4$  are not coplanar, we have that $x_1, x_2$ and $x_3$ are not collinear. This implies that the circle $\pa B_{r_1}(x_1) \cap \pa B_{r_2}(x_2)$ and the sphere $\pa B_{r_3}(x_3)$ intersect at two points $\{A, B\}$. If $A=B$, then $A$ is exactly the position $z_1$ we are looking for. Otherwise, $|x_4-A|\neq |x_4-B|$ because $x_4$ is not in the plane passing through $x_1, x_2$ and $x_3$. Thus $z_1=A$ if $r_4=|x_4-A|$, or else $z_1=B$.
\end{proof}

Note that the second equality in \eqref{11} does not hold in $\R^2$, and therefore, the fourth result in Theorem \ref{1k1M} is not clear in two dimensions. 
Actually, we claim that the modulus $|H^{(1)}_0(t)|$ is monotonous with respect to the variable $t$.
Numerical experiments indicate that this is indeed the case but a rigorous proof is not known. If this is correct, we can show that the distance $|x-z|$ in $\R^2$ can be uniquely determined by the modulus $|u^s(x, k)|$ of the scattered field, and therefore the location $z$ can be uniquely determined by the phaseless scattered fields  $|u^s(x, k)|$ at three sensors that are not collinear. This procedure is based on the phase retrieval technique proposed in the recent works \cite{JiLiu-elastic,JiLiu-electromagnetic,JiLiuZhang-source}.

To remove
the assumption on $|\tau_1|$, we can determine both the location $z_1$ and the strength $\tau_1$ by measurements with frequency in a bounded band as given in \eqref{kassumption}.

\begin{theorem}\label{mk1M}
For all $k\in (k_{-}, k_{+})$ and let $M=1$. Then we have the following uniqueness results.
\begin{itemize}
  \item In $\R^3$, assume that $u^s(x,k_{-})\neq u^s(x, k_{+})$ at four sensors $x\in\{x_1, x_2, x_3, x_4\}\subset\R^3\ba\{z_1\}$, which are not coplanar. Then both the location $z_1$ and the strength $\tau_1$ can be uniquely determined by the multi-frequency scattered fields $u^s(x,k), \,x\in\{x_1, x_2, x_3, x_4\},\,k\in (k_{-}, k_{+})$.
  \item In $\R^n$, assume that $u^\infty(\hx,k_{-})\neq u^\infty(\hx, k_{+})$ at $n$ linearly independent observation directions $\hx\in \{\hx_1, \cdots, \hx_n\}, \,n=2,3$. Then both the location $z_1$ and the strength $\tau_1$ can be uniquely determined by the multi-frequency far field patterns $u^\infty(\hx,k), k\in (k_{-}, k_{+})$ at $n$ linearly independent observation directions $\hx\in \{\hx_1, \cdots, \hx_n\}, \,n=2,3$.
\end{itemize}
\end{theorem}
\begin{proof}
By the representation \eqref{1Musuinf} for the scattered field, we have
\ben
\frac{u^s(x,k)}{u^s(x,k_{-})} = e^{i(k-k_{-})|x-z_1|}, \quad x\in \{x_1, x_2, x_3, x_4\}, \, k\in (k_{-}, k_{+}).
\enn
Taking integral on both sides with respect to the frequency $k$ over the frequency band $(k_{-}, k_{+})$, we have
\ben
\int_{k_{-}}^{k_{+}}\frac{u^s(x,k)}{u^s(x,k_{-})}dk
&=& \int_{k_{-}}^{k_{+}}e^{i(k-k_{-})|x-z_1|}dk\cr
&=& \frac{1}{i|x-z_1|}\left[e^{i(k_{+}-k_{-})|x-z_1|}-1\right]\cr
&=& \frac{-i}{|x-z_1|}\left[\frac{u^s(x,k_{+})}{u^s(x,k_{-})}-1\right], \quad x\in \{x_1, x_2, x_3, x_4\}.
\enn
This implies that
\be\label{x-z1}
|x-z_1| = -i\frac{\frac{u^s(x,k_{+})}{u^s(x,k_{-})}-1}{\int_{k_{-}}^{k_{+}}\frac{u^s(x,k)}{u^s(x,k_{-})}dk}, \quad x\in \{x_1, x_2, x_3, x_4\}.
\en
Note that our assumption on the scattered fields ensures that both the numerator and the denominator of the right hand side are nonzeros.
Thus $\tau_1$ can be uniquely recovered with the help of the representation \eqref{1Musuinf}. Since the distances $|x-z_1|$, $x\in \{x_1, x_2, x_3, x_4\}$, are uniquely determined, one can recover the location $z_1$ by a constructive way as the arguments in the previous Theorem \ref{1k1M}.

Now we turn to the far field measurements. Similarly, by the representation \eqref{1Musuinf} for the far field pattern, we have
\ben
\frac{u^\infty(\hx,k)}{u^\infty(\hx,k_1)} = e^{-i(k-k_1)\hx\cdot z_1}, \quad \hx\in \{\hx_1, \cdots, \hx_n\}, \, k\in (k_{-}, k_{+}).
\enn
Multiplying this identity by $-i\hx\cdot z_1$, integrating over the frequency band $(k_{-}, k_{+})$, we obtain
\ben
-i\hx\cdot z_1\int_{k_{-}}^{k_{+}}\frac{u^\infty(\hx,k)}{u^\infty(\hx,k_{-})}dk
&=& -i\hx\cdot z_1\int_{k_{-}}^{k_{+}}e^{-i(k-k_{-})\hx\cdot z_1}dk\cr
&=& e^{-i(k_{+}-k_{-})\hx\cdot z_1}-1\cr
&=& \frac{u^\infty(\hx,k_{+})}{u^\infty(\hx,k_{-})}-1, \quad\hx\in \{\hx_1, \cdots, \hx_n\}.
\enn
This implies that
\be\label{hxdotz1}
\hx\cdot z_1 = i\frac{\frac{u^\infty(\hx,k_{+})}{u^\infty(\hx,k_{-})}-1}{\int_{k_{-}}^{k_{+}}\frac{u^\infty(\hx,k)}{u^\infty(\hx,k_{-})}dk}, \quad \hx\in \{\hx_1, \cdots, \hx_n\}.
\en
The scattering strength $\tau_1$ is then uniquely determined by combining the representation \eqref{1Musuinf}. The identify \eqref{hxdotz1} also implies that
$z_1$ is uniquely determined by noting the fact that the observation directions $\hx_1, \hx_2$ and  $\hx_n$ are linearly independent.
\end{proof}

Difficulties arise if there are more than one point source, i.e., $M>1$. This is due to the severe nonlinearity between the measurements and the locations of the point sources. Note that
\be\label{hyperplane}
\Pi_{l,m}:\quad \hx_l\cdot(z-z_m)=0
\en
is the hyperplane passing through the location $z_m, \, m=1, 2, \cdots, M,$ with normal $\hx_l\in\Theta_{L},\, l=1, 2, \cdots, L$.
For any point $z\in\R^n$, denote by $f(z)$ the number of the hyperplanes $\Pi_{l,m},\,l=1, 2, \cdots, L,\,m=1, 2, \cdots, M,$ passing through $z$.

\begin{lemma}\label{lemma1}
We consider $M$ points $z_1, z_2,\cdots, z_M$ in $\R^n$.  Define
\be\label{L}
L:=\left\{
     \begin{array}{ll}
       M+1, & \hbox{in $\R^2$\rm ;} \\
       2M+1, & \hbox{in $\R^3$.}
     \end{array}
   \right.
\en
Recall the sparse observation directions set $\Theta_{L}:=\{\hx_1, \hx_2, \cdots, \hx_L\}$
and the hyperplanes $\Pi_{l,m}$ as in \eqref{hyperplane} for $l=1, 2, \cdots, L$ and $m=1, 2, \cdots, M$.
\begin{itemize}
  \item In $\R^2$, if any two directions in $\Theta_{L}$ are not collinear, then $f(z_m)=M+1$\, $m=1, 2, \cdots, M$ and $f(z)\leq M$ if $z\in\R^2\ba\{z_1, z_2,\cdots, z_M\}$.
  \item In $\R^3$, if any three directions in $\Theta_{L}$ are not coplanar, then $f(z_m)=2M+1$\, $m=1, 2, \cdots, M$ and $f(z)\leq 2M$ if $z\in\R^3\ba\{z_1, z_2,\cdots, z_M\}$.
\end{itemize}
\end{lemma}
\begin{proof}
In $\R^2$, since any two directions in $\Theta_{L}$ are not collinear, we obtain $L=M+1$ hyperplanes $\Pi_{l,m},\, l=1,2,\cdots, N$ passing through $z_m, \,m=1,2,\cdots,M$. Thus $f(z_m)=M+1$\, $m=1, 2, \cdots, M$. For any $z\in\R^2\ba\{z_1, z_2,\cdots, z_M\}$, the value $f(z)$ increase only if there is a hyperplane passing through $z$ and some $z_m, m=1, 2, \cdots, M$ simultaneously. There are only $M$ given points $z_m, m=1, 2, \cdots, M$. Therefore, for any $z\in\R^2\ba\{z_1, z_2,\cdots, z_M\}$, there are at most $M$ hyperplanes passing through $z$. This implies $f(z)\leq M$ if $z\in\R^2\ba\{z_1, z_2,\cdots, z_M\}$.

In $\R^3$, since any three directions in $\Theta_{L}$ are not coplanar, we obtain $L=2M+1$ hyperplanes $\Pi_{l,m},\, l=1,2,\cdots, L$ passing through $z_m, \,m=1,2,\cdots,M$. Thus $f(z_m)=L$,\, $m=1, 2, \cdots, M$. It is clear that $f(z)\leq L$ for all $z\in\R^3$. Assume that there exists a point $z^\ast\in\R^3\ba\{z_1, z_2,\cdots, z_M\}$ such that $f(z^\ast)=L$. Then for each observation direction $\hx_l$, there exists some $z_m\in \{z_1, z_2,\cdots, z_M\}$ such that
$\hx_l\cdot (z^\ast-z_m)=0,\, l=1,2,\cdots, 2M+1$.  By the pigeonhole principle, there exists one point $z_m\in \{z_1, z_2,\cdots, z_M\}$ such that for three different observation directions $\hx_{l_1}, \,\hx_{l_2}$ and $\hx_{l_3}$ such that $\hx\cdot (z^\ast-z_m)=0$ for $\hx\in \{\hx_{l_1}, \,\hx_{l_2},\,\hx_{l_3}\}$.
Therefore the observation directions $\hx_{l_1}, \,\hx_{l_2}$ and $\hx_{l_3}$ are coplanar. This leads to a contradiction to our assumption on the observation directions. This completes the proof.

\end{proof}

With the results given in Lemma \ref{lemma1}, we can determine the locations and scattering strengths of the point sources by the multi-frequency far field patterns at finitely many observation directions.

\begin{theorem}\label{Uniqueness-nknM}
We consider $M$ isolated point sources with locations $z_m\in\R^n$ and scattering strengths $\tau_m\in\C\ba\{0\},\, m=1,2,\cdots, M$. Recall $L$ defined by \eqref{L}
and the observation direction set $\Theta_L:=\{\hx_1, \hx_2, \cdots, \hx_L\}\subset \Sp^{n-1}$.
Consider the same assumptions on the observation directions as in the previous Lemma \ref{lemma1}.
Then we have the following uniqueness results.
\begin{itemize}
  \item The locations $z_m$ and scattering strengths $\tau_m,\, m=1,2,\cdots, M$ can be uniquely determined by the far field patterns $u^{\infty}(\pm\hx, k)$ for all $\hx\in\Theta_L$ and $k\in (k_{-}, k_{+})$.
  \item If we further assume that $\tau_m\in \R$, then locations $z_m$ and scattering strengths $\tau_m,\, m=1,2,\cdots, M$ can be uniquely determined by the far field patterns $u^{\infty}(\hx, k)$ for all $\hx\in\Theta_L$ and $k\in (k_{-}, k_{+})$.
\end{itemize}
\end{theorem}
\begin{proof}
Note that the far field pattern $u^{\infty}(\hx, k) = \sum_{m=1}^{M}\tau_m e^{-ik\hx\cdot z_m}$ depends analytically on $k$, thus we have the far field patterns for all frequencies in $(0, \infty)$. Integrating with respect to $k$, we deduce that
\ben
\int_{0}^{\infty} \Big(u^{\infty}(\hx, k) + u^{\infty}(-\hx, k)\Big)dk
&=& \int_{0}^{\infty} \sum_{m=1}^{M}\tau_m\Big(e^{-ik\hx\cdot z_m} + e^{ik\hx\cdot z_m}\Big)dk \cr
&=& \sum_{m=1}^{M}\tau_m\int_{-\infty}^{\infty}e^{-ik\hx\cdot z_m} dk \cr
&=& 2\pi\sum_{m=1}^{M}\tau_m \delta(\hx\cdot z_m), \quad \hx\in \Theta_L,
\enn
where $\delta$ is the Dirac delta function. This implies for each $\hx\in \Theta_L$, the values $\hx\cdot z_m,\,m=1,2,\cdots, M$ are given uniquely by the far field patterns $u^{\infty}(\pm\hx, k)$ for all $\hx\in\Theta_L$ and $k>0$. From this, we can then define the hyperplanes $\Pi_{l,m}: \,\hx_l\cdot (z-z_m)=0,\, l=1,2,\cdots, L,\, m=1,2,\cdots,M$. Using Lemma \ref{lemma1}, we deduce that the locations $z_m,\,m=1,2,\cdots, M$ can be uniquely recovered by the far field patterns $u^{\infty}(\pm\hx, k)$ for all $\hx\in\Theta_L$ and $k>0$.
For any $z_{m^\ast}\in \{z_1, z_2, \cdots, z_M\}$, by the assumption on the observation directions, we can always choose some $\hx_l\in\Theta_L$ such that
$\hx_l\cdot (z_{m^\ast}-z_m)\neq 0$ for all $z_m\neq z_{m^\ast}, \,m=1,2,\cdots, M$. Then by the representation of the far field pattern, we have
\ben
u^{\infty}(\hx_l, k)e^{ik\hx_l\cdot z_{m^\ast}} = \tau_{m^\ast} +\sum_{m=1, m\neq {m^\ast}}^{M} \tau_m e^{-ik\hx_l\cdot(z_m-z_{m^\ast})}.
\enn
For any $K>0$,
\ben
&&\int_{0}^{K}\Big(u^\infty(\hx_l, k)e^{ik\hx_l\cdot z_{m^\ast}} +u^{\infty}(-\hx_l, k)e^{-ik\hx_l\cdot z_{m^\ast}}\Big)dk\cr
&=&2K\tau_{m^\ast}+\sum_{m=1, m\neq {m^\ast}}^{M} \tau_m \int_{-K}^{K}e^{-ik\hx_l\cdot (z_m-z_{m^\ast})}dk.
\enn
Letting $K\rightarrow\infty$, we obtain that the second term  on the right hand side of the above equality tends to $2\pi \sum_{m=1, m\neq {m^\ast}}^{M} \tau_m \delta(\hx_l\cdot(z_m-z_{m^\ast}))$ and therefore vanishes since $\hx_l\cdot (z_{m^\ast}-z_m)\neq 0$ for all $z_m\neq z_{m^\ast}, \,m=1,2,\cdots, M$.
This implies
\ben
\tau_m = \lim_{K\rightarrow\infty} \frac{1}{2K}\int_{0}^{K}\Big(u^{\infty}(\hx_l, k)e^{ik\hx_l\cdot z_m} +u^{\infty}(-\hx_l, k)e^{-ik\hx_l\cdot z_m}\Big)dk.
\enn
In other words, the scattering strengths $\tau_m, \,m=1,2,\cdots, M$ are also uniquely determined by the far field patterns $u^{\infty}(\pm\hx, k)$ for all $\hx\in\Theta_L$ and $k>0$. This completes the proof for the first statement.

If we have the a priori information that $\tau_m\in\R$, then we need less data as given in the second statement. In this case, we define
\ben
u^{\infty}(\hx, k) := \ov{u^{\infty}(\hx, -k)}, \quad k<0,
\enn
where by $\ov{\cdot}$ we denote the complex conjugate. Then the proof follows by similar arguments.
\end{proof}

Inspired by the arguments in \cite{GriesmaierSchmiedecke}, uniqueness can also be established by finitely many properly chosen frequencies.
For a fixed observation direction $\hx\in\Theta_L$, assume that the far field patterns $u^{\infty}(\hx, k_j)$ in \eqref{uinf-source} are given for $J$ equidistant wavenumbers
\be\label{kj}
k_j = j k_{min},\quad j=1,2,\cdots, J,
\en
where
\be\label{kmin}
0<k_{min}\leq\frac{\pi}{2R} \quad\mbox{and}\quad J>2M.
\en
Here, $R>0$ denotes the radius of the smallest ball centered at the origin that contains the locations of the point sources.
The upper bound on $k_{min}$ in \eqref{kmin} implies that $|k_{min}\hx\cdot z_m|<\pi$ for all $m=1,2,\cdots, M$. This further implies that
the value $\hx\cdot z_m$ is uniquely determined by $e^{-ik_{min}\hx\cdot z_m}$.
In the following we develop a rigorous characterization of the projections $\hx\cdot z_1, \,\hx\cdot z_2, \cdots, \hx\cdot z_M$
of the locations $z_1, z_2, \cdots, z_M$ of the point sources from the far field patterns $u^{\infty}(\hx, k_j), j=1,2,\cdots, J$.

For a fixed observation direction $\hx$ different locations $z_{m_1}\neq z_{m_2}$ may yield the same projections $\hx\cdot z_{m_1}=\hx\cdot z_{m_2}$,
and the corresponding summands in \eqref{uinf-source} even cancel if $\tau_{m_1}=-\tau_{m_2}$. In this case, the point sources at $z_{m_1}$ and $z_{m_2}$ would not
contribute to the far field data $u^{\infty}(\hx, k_j), j=1,2,\cdots, J$, and consequently they can not be reconstructed from such far field data.
We introduce the set and its cardinality
\ben
\mathcal {M}_{\hx}:={\rm supp}\left(\sum_{m=1}^{M}\tau_m \delta(\hx\cdot z_m)\right) \subset\R, \quad M^{\ast}:=|\mathcal {M}_{\hx}|.
\enn
Clearly, $M^\ast\leq M$.
Accordingly, we rewrite the far field pattern as
\be
u^{\infty}(\hx, k_j)= \sum_{m=1}^{M^\ast}\tau_m^{\ast}\xi^{j}_{m},\quad j=1,2,\cdots, J,
\en
where $\xi_m:=e^{-ik_{min}f_m}$ for any $f_m\in\mathcal {M}_{\hx}$. The far field patterns $u^{\infty}(\hx, k_j), j=1,2,\cdots, J$, define the Hankel matrix
\ben
U =\left(
     \begin{array}{cccc}
       u^{\infty}(\hx, k_1) & u^{\infty}(\hx, k_2) & \cdots & u^{\infty}(\hx, k_{M+1}) \\
       u^{\infty}(\hx, k_2) & u^{\infty}(\hx, k_3) & \cdots & u^{\infty}(\hx, k_{M+2}) \\
       \vdots & \vdots & \ddots & \vdots \\
       u^{\infty}(\hx, k_{J-M}) & u^{\infty}(\hx, k_{J-M+1}) & \cdots & u^{\infty}(\hx, k_{J}) \\
     \end{array}
   \right)\in\C^{(J-M)\times(M+1)}.
\enn

\begin{lemma}\label{Uproperty}
The Hankel matrix $U$ has a factorization of the form
\be\label{Ufac}
U=V_{J-M}DV^{\top}_{M+1},
\en
where $V_\mathbbm{i}=(v_1, v_2, \cdots, v_{M^{\ast}})\in\C^{\mathbbm{i}\times M^\ast}, \mathbbm{i}\geq2$, denotes a Vandermonde matrix with
$v_m=(1, \xi_m, \xi^2_m,\cdots,\xi^{\mathbbm{i}-1}_m)^{\top}$ for $m=1, 2, \cdots, M^\ast$, and the matrix $D=diag(\tau^\ast_m \xi_m)\in\C^{M^\ast\times M^\ast}$.
Furthermore,
\be\label{rank}
{\rm rank} (U) = {\rm rank} (V_{J-M}) = {\rm rank} (V_{M+1}) = M^\ast
\en
and
\be\label{RangeIdentity}
\mathcal {R} (U) = \mathcal {R} (V_{J-M}).
\en
\end{lemma}
\begin{proof}
The factorization \eqref{Ufac} follows by a straightforward calculation. Since $\xi_1, \xi_2, \cdots, \xi_{M^\ast}$ are mutually distinct by construction, the rank
of the Vandermonde matrix $V_\mathbbm{i}\in\C^{\mathbbm{i}\times M^\ast}$ satisfies
\ben
{\rm rank} (V_\mathbbm{i}) = \min \{\mathbbm{i}, M^\ast\}.
\enn
Therefore, we deduce from the assumption $J > 2M \geq 2M^\ast$ that
\ben
{\rm rank} (V_{J-M}) = {\rm rank} (V_{M+1}) = M^\ast.
\enn
Clearly, $D\in \C^{M^\ast\times M^\ast}$ is invertible, and thus the factorization \eqref{Ufac} implies
\ben
{\rm rank} (U) = M^\ast.
\enn

Finally, we show the range identity \eqref{RangeIdentity}. It is clear that $\mathcal {R} (U) \subset \mathcal {R} (V_{J-M})$ from the factorization \eqref{Ufac}.
Conversely, assume that
\ben
\phi = V_{J-M}\psi
\enn
for some $\psi\in \C^{M^\ast\times 1}$. Using the invertibility of $D$ again, we denote by
\ben
\psi^\ast:= D^{-1}\psi.
\enn
Note that ${\rm rank} (V_{M+1}) = M^\ast$, we can always find some $\eta\in \C^{(M+1)\times 1}$ such that
\ben
\psi^\ast = V^{\top}_{M+1}\eta.
\enn
Combining the previous identities, we deduce that $ U\eta = V_{J-M}DV^{\top}_{M+1} \eta= V_{J-M}D \psi^\ast = V_{J-M}\psi = \phi$, i.e., $\phi\in \mathcal {R} (U)$.
This finishes the proof.
\end{proof}

\begin{theorem}\label{uni-hyperplanes}
For a single observation direction $\hx\in \Sp^{n-1}$, the hyperplanes $\Pi_m:\quad \hx\cdot(z-z_m)=0, m=1, 2, \cdots, M$, are uniquely determined by the far field patterns $u^{\infty}(\hx, k_j), j=1,2,\cdots, J$.
\end{theorem}
\begin{proof}
Let $B_R$ be the ball centered at the origin with radius $R$. For a point $z\in B_R\subset \R^n$, define $\xi:= e^{-ik_{min}\hx\cdot z}$ and let $\phi_z:=(1, \xi, \xi^2,\cdots,\xi^{J-M-1})^{\top}\in \C^{(J-M)\times 1}$.
We claim that $\phi_z\in \mathcal {R} (U)$ if and only if $\xi\in \{\xi_1, \xi_2, \cdots, \xi_{M^\ast}\}$.

Let first $\xi\in \{\xi_1, \xi_2, \cdots, \xi_{M^\ast}\}$. Then clearly $\phi_z\in \mathcal {R} (V_{J-M})$ by the construction of the matrix $V_{J-M}$. The range identity \eqref{RangeIdentity} in Lemma \ref{Uproperty} implies that $\phi_z\in \mathcal {R} (U)$.

Let now $\xi\notin \{\xi_1, \xi_2, \cdots, \xi_{M^\ast}\}$. The assumption $J>2M$ implies that $J-M>M+1>M^\ast + 1$. Then the novel Vandermonde matrix $(V_{J-M}, \phi_z)\in\C^{(J-M)\times(M^\ast+1)}$ has rank $M^{\ast}+1$. This implies that $\phi_z\notin \mathcal {R} (V_{J-M})$ and consequently $\phi_z\notin \mathcal {R} (U)$
by using the range identity \eqref{RangeIdentity} again.

If $\phi_z\in \mathcal {R} (U)$, then $\xi= \xi_m$ for some $1\leq m\leq M$, i.e., $e^{-ik_{min}\hx\cdot z} = e^{-ik_{min}\hx\cdot z_m}$. By the assumption \eqref{kmin} of the smallest wavenumber $k_{min}$, we deduce that
$\hx\cdot z = \hx\cdot z_m$.
That is, the hyperplane $\Pi_m:\quad \hx\cdot(z-z_m)=0$ is uniquely determined.
\end{proof}

Combining Theorem \ref{uni-hyperplanes} and Lemma \ref{lemma1}, we immediately have the following uniqueness results on the determination of the locations by the
far field patterns with finitely many observation directions and finitely many frequencies.

\begin{theorem}
Let $k_j$ be the wave number given by \eqref{kj} satisfying \eqref{kmin}. Let $L$ be given as in \eqref{L}, we consider $L$ observation directions $\hx_l, l=1,2,\cdots, L$ satisfying the same condition as in Lemma \ref{lemma1}.  Then the locations $z_1, z_2,\cdots, z_M$ can be uniquely determined by the far field patterns $u^{\infty}(\hx_l, k_j)$, $l=1,2,\cdots, L, \,j=1,2,\cdots, J$.
\end{theorem}

\begin{theorem}
Let $k_j$ be the wave number given by \eqref{kj} satisfying \eqref{kmin}. Let $\hx\in \Sp^{n-1}$ be such that $\hx\cdot(z_{m_1}-z_{m_2})\neq 0$ if $m_1\neq m_2$.  Given the locations $z_1, z_2,\cdots, z_M$, the corresponding scattering strengths $\tau_1, \tau_2,\cdots, \tau_M$ can be uniquely determined by the multi-frequency far field patterns $u^{\infty}(\hx, k_j)$, $j=1,2,\cdots, J$ at the fixed observation direction $\hx$.
\end{theorem}
\begin{proof}
Recall that the far field patterns are given by
\be\label{uinfkj}
u^{\infty}(\hx, k_j)= \sum_{m=1}^{M}\tau_m\eta^{j}_{m},\quad j=1,2,\cdots, J,
\en
where $\eta_m:=e^{-ik_{min}\hx\cdot z_m}$. We rewrite the equations \eqref{uinfkj} in the matrix form
\be\label{tauequ}
\mathbb{V}\mathbb{T}=\mathbb{U},
\en
where $\mathbb{T}=(\tau_1, \tau_2, \cdots, \tau_M)^{\top}$, $\mathbb{V}= (\mathbbm{v}_1, \mathbbm{v}_2, \cdots, \mathbbm{v}_M)\in \C^{J\times M}$ denotes a Vandermonde matrix with $\mathbbm{v}_m=(\eta_m, \eta_m^2, \cdots, \eta_m^{J})^{\top}$, $m=1,2,\cdots, M$ and $\mathbb{U}=(u^{\infty}(\hx, k_1), u^{\infty}(\hx, k_2), \cdots, u^{\infty}(\hx, k_J))^{\top}$.
Since $\eta_1, \eta_2, \cdots, \eta_M$ are distinct by the assumption on $\hx$, we have that
\ben
{\rm rank}(\mathbb{V})= {\rm min}\{J, M\}=M.
\enn
Therefore, the equation \eqref{tauequ} is uniquely solvable.
In other words, the scattering strengths $\tau_1, \tau_2,\cdots, \tau_M$ can be uniquely determined by the far field patterns $u^{\infty}(\hx, k_j)$, $j=1,2,\cdots, J$. The proof is complete.
\end{proof}
\quad\\

Finally, we turn to the scattering of plane waves by point like scatterers. Recall that the scattered field and the corresponding far field pattern are given by
\ben
u^{s}(x,\hth,k)= \sum_{m=1}^{M}\tau_m u^{i}(z_m,\hth,k)\Phi_k(x,z_m), \quad x\in\R^n\ba\{z_1, z_2, \cdots, z_M\}
\enn
and
\ben
u^{\infty}(\hx, \hth, k)= \sum_{m=1}^{M} \tau_m e^{-ikz_m\cdot(\hx-\hth)},\quad \hx,\hth\in \Sp^{n-1},
\enn
respectively. We collect the uniqueness results in the following theorem. We omit the proof since it is similar to the case of inverse source scattering problems.

\begin{theorem}
We consider the scattering of plane waves by point like scatterers. Then we have the following uniqueness results.
\begin{description}
  \item[M=1] \begin{itemize}
  \item Let $k>0$ be fixed. For any single sensor $\hx\in\Sp^{n-1}$ or $x\in\R^3\ba\{z_1\}$, we have
     \be\label{11-obstacle}
     \tau_1=u^\infty(\hx, \hth, k) e^{ikz_1\cdot(\hx-\hth)}, \qquad |x-z_1|=\frac{|\tau_1|}{4\pi|u^s(x,\hth, k)|}.
     \en
  \item Let $k>0$ be fixed. If we know the location $z_1$ in advance, then the strength $\tau_1$ is uniquely determined by the scattered field $u^s(x,\hth, k)$ at a single sensor $ x\in\R^n\ba\{z_1\}$ or the far field pattern $u^\infty(\hx,\hth, k)$ at a single observation direction $\hx\in\Sp^{n-1}$.
  \item Let $k>0$ be fixed. If we know the modulus $|\tau_1|$ in advance, then the location $z_1\in \R^3$ is uniquely determined by the phaseless scattered field $|u^s(x,\hth, k)|$ at four sensors $x\in\{x_1, x_2, x_3, x_4\}\subset\R^3\ba\{z_1\}$, which are not coplanar.
  \item In $\R^3$, both the location $z_1$ and the strength $\tau_1$ can be uniquely determined by the multi-frequency scattered fields $u^s(x,\hth, k), k\in (k_{-}, k_{+})$ at four sensors $x\in\{x_1, x_2, x_3, x_4\}\subset\R^3\ba\{z_1\}$, which are not coplanar.
  \item In $\R^n$, both the location $z_1$ and the strength $\tau_1$ can be uniquely determined by the multi-frequency far field patterns $u^\infty(\hx,\hth,k), k\in (k_{-}, k_{+})$ at $n$ pairs of linearly independent directions $\phi_j:=\hx_j-\hth_j, \,j=1,2,\cdots, n$.
      \end{itemize}
  \item[M>1] Recall $L$ and $\Theta_L$ defined by \eqref{L} and \eqref{thetaL}, respectively. Then we have the following results.
    \begin{itemize}
  \item The locations $z_m$ and scattering strengths $\tau_m,\, m=1,2,\cdots, M$ can be uniquely determined by the far field patterns $u^{\infty}(\hx, -\hx, k)$ and $u^{\infty}(-\hx, \hx, k)$ for all $\hx\in\Theta_L$ and $k\in (k_{-}, k_{+})$.
  \item If we further assume that $\tau_m\in \R$, then locations $z_m$ and scattering strengths $\tau_m,\, m=1,2,\cdots, M$ can be uniquely determined by the far field patterns $u^{\infty}(\hx, -\hx, k)$ for all $\hx\in\Theta_L$ and $k\in (k_{-}, k_{+})$.
  \item We consider the following wave numbers
  \be\label{kj}
k_j = j k_{min},\quad j=1,2,\cdots, J,
\en
where
\be\label{kmin}
0<k_{min}\leq\frac{\pi}{4R} \quad\mbox{and}\quad J>2M.
\en
Here, $R>0$ denotes the radius of the smallest ball centered at the origin that contains the locations of the point like scatterers. Then the locations $z_m$, $m=1,2,\cdots, M$ can be uniquely determined by the backscattering far field patterns $u^{\infty}(\hx, -\hx, k_j), j=1,2,\cdots, J$ at finitely many observation directions $\hx\in\Theta_L$. Furthermore, taking $\hx\in\Sp^{n-1}$ such that $\hx\cdot z_{m_1}\neq \hx\cdot z_{m_2}$ if $m_1\neq m_2$, then the scattering strengths $\tau_m, m=1,2,\cdots, M$ can be uniquely determined by the far field patterns far field patterns $u^{\infty}(\hx, -\hx, k_j), j=1,2,\cdots, J$ at the fixed observation directions $\hx\in\Sp^{n-1}$.
    \end{itemize}
\end{description}
\end{theorem}

\section{Numerical methods}
\label{NumMethods}
This section aims to introduce some novel numerical methods for identifying the locations and scattering strengths of the point object.
Some of the numerical methods are originated from the constructive uniqueness proof in the previous section.

\subsection{Numerical methods for Point sources}
We begin with the scattering by point sources. If $M=1$, given $|\tau_1|$, we introduce the following indicator to locate the position of $z_1$.
\be\label{Iz1}
I_{1}(z):=\frac{1}{\sum_{j=1}^{4}\Big|\frac{|x_j-z|}{|\tau_1|}-\frac{1}{4\pi|u^s(x_j, k)|}\Big|}
\en
where $x_j,\,j=1,2,3,4$ are four points in $\R^3$ such that they are not coplanar.
By the representation of the scattered field given in  \eqref{us-source}, we know that $I_{1}(z)\rightarrow \infty$ if $z\rightarrow z_1$. Thus $I_{1}(z)$ blows up at $z=z_1$.

The proof of Theorem \ref{1k1M} gives a constructive way to look for the the position of $z_1$. Based on this, we introduce a simple geometrical method to reconstruct a point in $\R^3$ from four distances to the given points. The procedure is actually an extension of the geometrical method to reconstruct a point from three distances to the given points in $\R^2$, which is the key idea of phase retrieval \cite{JiLiu-elastic,JiLiu-electromagnetic,JiLiuZhang-source}.
However, the procedure for reconstruct a point in $\R^3$ is much more technical.

Let $x_j, j=1,2,3,4$ be four given points in $\R^3$  such that they are not coplanar, denote by $r_j:=|x_j-z|$. The following scheme provides a constructive way for determining the unknown point $z$.

\begin{figure}[htbp]
  \centering
  \subfigure[\textbf{Step one.}]{
    \includegraphics[height=1.3in,width=1.6in]{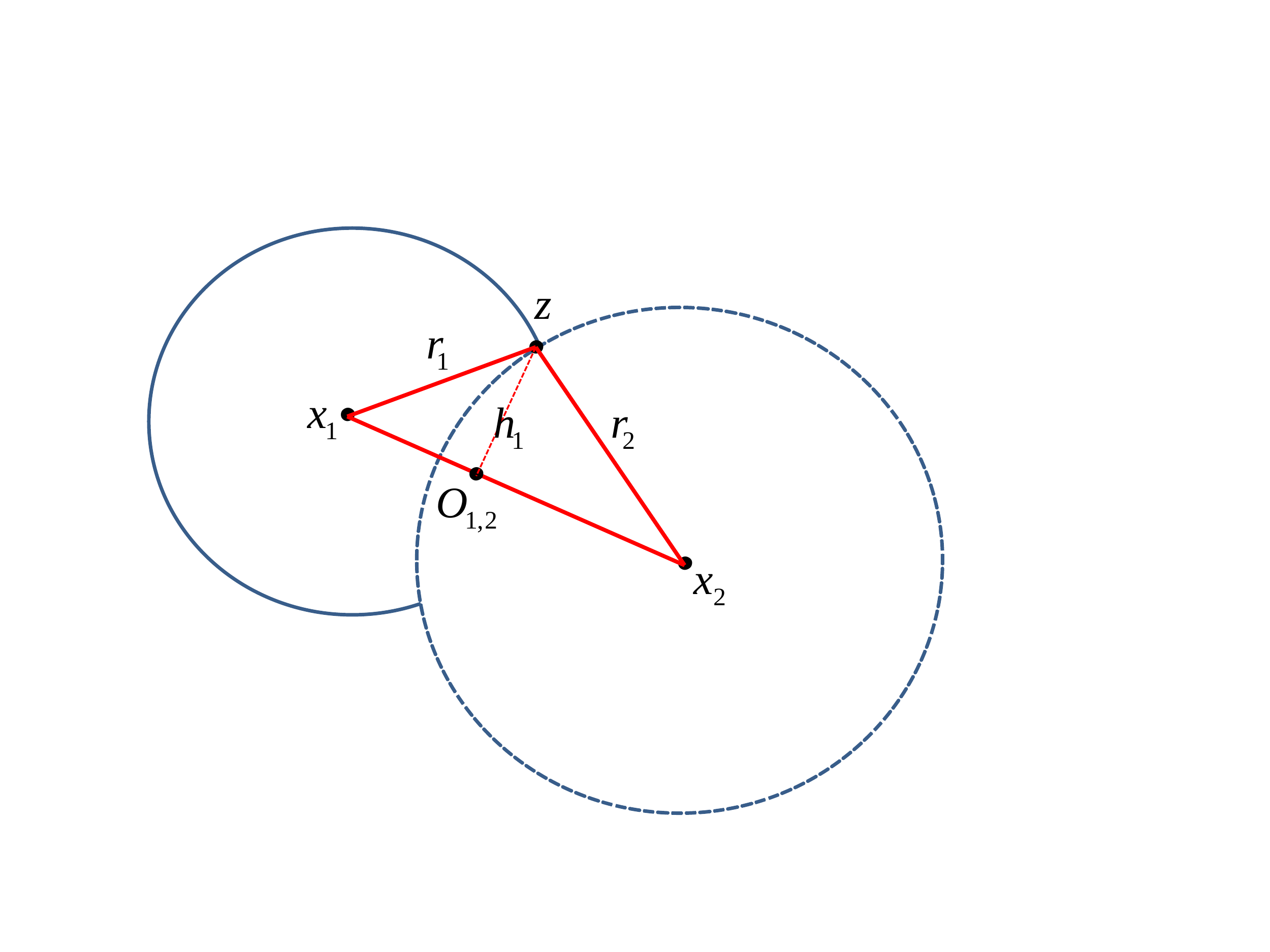}}\qquad\qquad
  \subfigure[\textbf{Step two.}]{
    \includegraphics[height=1.6in,width=1.6in]{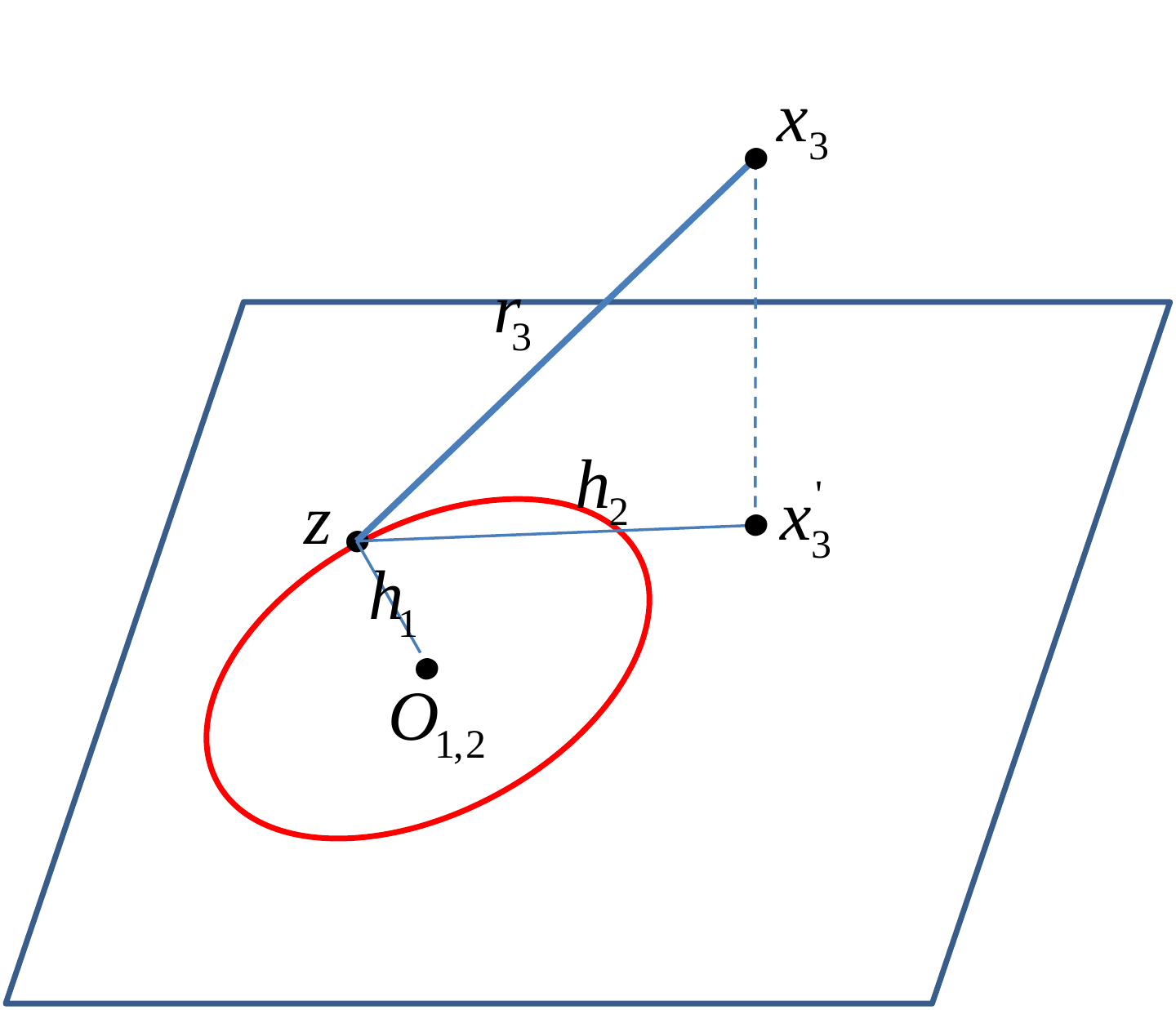}}\qquad\qquad
  \subfigure[\textbf{Step three.}]{
    \includegraphics[height=1.3in,width=1.6in]{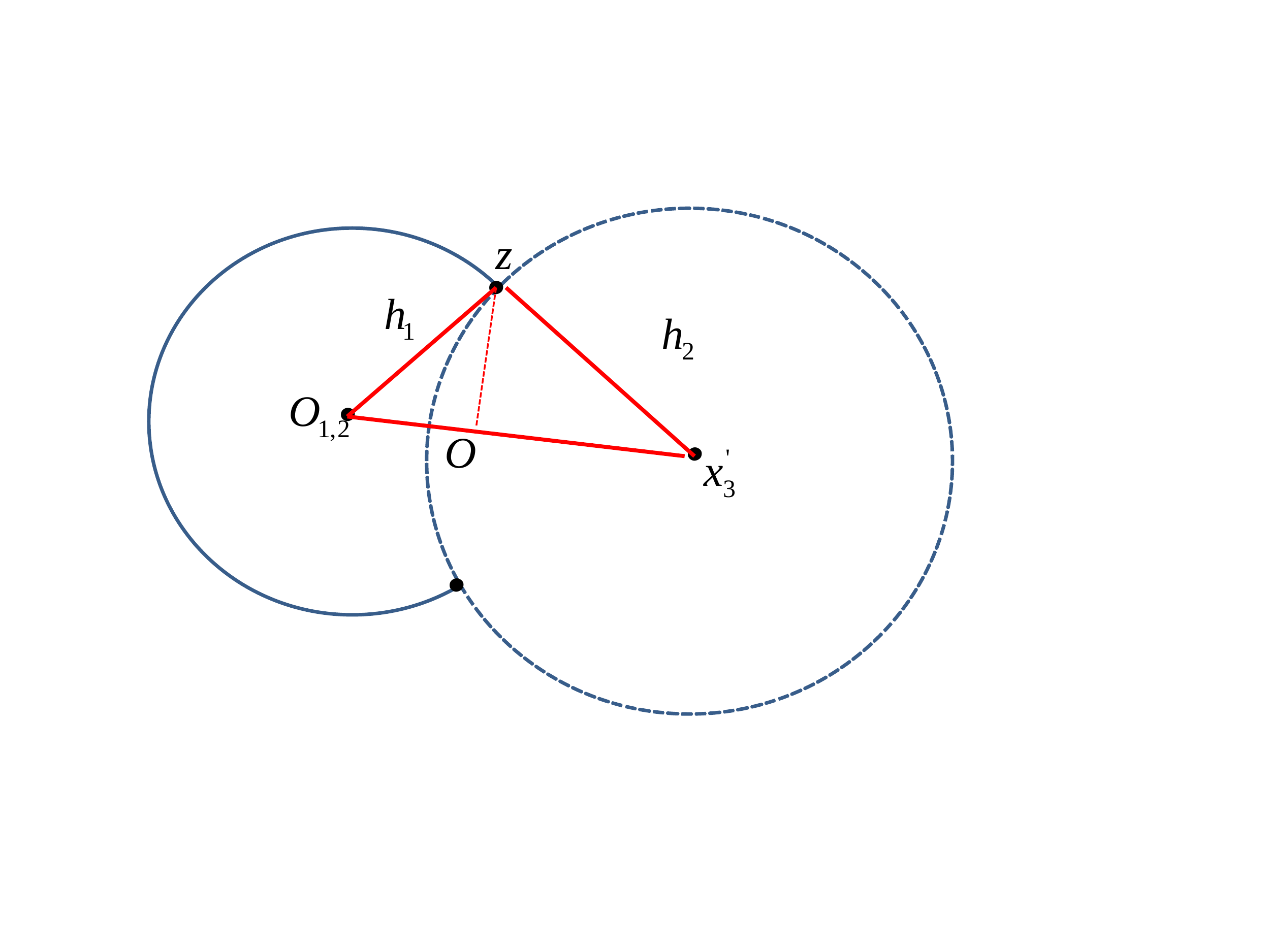}}
\caption{\bf Scheme for determining a point in $\R^3$.}
\label{scheme}
\end{figure}

{\bf One point determination scheme: locating a point in $\R^3$ by four distances to the given points:}
\begin{description}{\em
  \item $(1)$ {\bf\rm Determine the perpendicular foot $O_{1,2}$ from two given points $x_j,\,j=1,2$ and two distances $r_j,\,j=1,2$.} \\
              As shown in Figure \ref{scheme}(a), by the Heron's formula, the area of the triangle $\Delta zx_1x_2$ is
              \ben
              S_{\Delta zx_1x_2}=\sqrt{p(p-r_1)(p-r_2)(p-|x_1-x_2|)}
              \enn
              with $p:=\frac{r_1+r_2+|x_1-x_2|}{2}$. Note that $S_{\Delta zx_1x_2}=\frac{1}{2}|x_1-x_2||z-O_{1,2}|$, we deduce that
              \ben
              h_1:=|z-O_{1,2}|=\frac{2S_{\Delta zx_1x_2}}{|x_1-x_2|}.
              \enn
              Then by the Pythagorean theorem, we have
              \ben
              |x_1-O_{1,2}| = \sqrt{r^2_1-h^2_1} = \sqrt{r^2_1-\frac{4S^2_{\Delta zx_1x_2}}{|x_1-x_2|^2}}.
              \enn
              Denote by $t_1:=\frac{|x_1-O_{1,2}|}{|x_1-x_2|}$, we have
              \ben
              O_{1,2}=\left\{
                        \begin{array}{ll}
                          x_1+t_1(x_2-x_1), & \hbox{if $r_1^2+|x_1-x_2|^2\geq r_2^2$;} \\
                          x_1-t_1(x_2-x_1), & \hbox{else.}
                        \end{array}
                      \right.
              \enn
  \item $(2)$ {\bf\rm Determination of $h_2$ and the projection $x_3^{\prime}$.} \\
             As shown in Figure \ref{scheme}(b), noting that $x_3x_3^{\prime}$ is parallel to $x_1x_2$, we have
             \ben
              x_3^{\prime}=x_3+t(x_1-x_2)
             \enn
             with $t$ is chosen such that $(x_3^{\prime}-O_{1,2})\cdot (x_1-x_2)=0$. By the Pythagorean theorem again, we have
             \ben
              h_2:=|z-x_3^{\prime}|=\sqrt{r_3^2-|x_3-x_3^{\prime}|^2}.
             \enn
  \item $(3)$ {\bf\rm Determination of $O$.} \\
             As shown in Figure \ref{scheme}(c), following the step one to reconstruct the point $O$.
  \item $(4)$ {\bf\rm Determine the point $z$.} \\
            Define
             \ben
             l:=\frac{(x_1-x_2)\times(x_3^{\prime}-O_{1,2})}{|(x_1-x_2)\times(x_3^{\prime}-O_{1,2})|}.
             \enn
             Then
             \be\label{z}
             z=O\pm t_2 l
             \en
             with $t_2:=\sqrt{h_1^2-|O-O_{1,2}|^2}$.
            Choosing $z$ from \eqref{z} by letting $|z-x_4|=r_4$.}
\end{description}

Actually, even without knowing $|\tau_1|$, we can determine the distance between the sensor $x$ and the unknown location $z_1$ by the formula \eqref{x-z1},
i.e.,
\be\label{Numx-z1}
|x-z_1| = -i\frac{\frac{u^s(x,k_{+})}{u^s(x,k_{-})}-1}{\int_{k_{-}}^{k_{+}}\frac{u^s(x,k)}{u^s(x,k_{-})}dk}, \quad x\in \{x_1, x_2, x_3, x_4\}.
\en
The price to pay is more data with respect to the frequency $k$. Then, one may derive a formula to locate $z_1$ with the help of a geometrical argument.
After determining the location $z_1$, the scattering strength $\tau_1$ can be computed directly by the representation \eqref{us-source} of the scattered field $u^s(x, k)$ at any sensor $x$.

If the far field pattern is considered, we may determine $\hx\cdot z_1$ by the formula \eqref{hxdotz1}, i.e.,
\be\label{Numhxdotz1}
\hx\cdot z_1 = i\frac{\frac{u^\infty(\hx,k_{+})}{u^\infty(\hx,k_{-})}-1}{\int_{k_{-}}^{k_{+}}\frac{u^\infty(\hx,k)}{u^\infty(\hx,k_{-})}dk}, \quad \hx\in \{\hx_1, \cdots, \hx_n\}.
\en
Then the location $z_1$ can be determined by solving a
system of linear equation with $n$ unknowns in $\R^n,\,n=2,3$. In particular, one may take $\hx$ to be the unit vectors in Cartesian coordinates. After this, the scattering strength can be computed by the representation \eqref{uinf-source} of the far field pattern.

We are more interested in the case with multiple point sources, i.e., $M>1$.
Define
\be\label{IzMx}
I_{M}(z, \hx):=\left|\int_{0}^{K} \left\{u^{\infty}(\hx,k) e^{ik\hx\cdot z}+u^{\infty}(-\hx,k) e^{-ik\hx\cdot z}\right\}dk\right|, \quad z\in\R^n, \hx\in\Theta_L.
\en
By the analysis in the proof of Theorem \ref{Uniqueness-nknM}, for large $K$, the indicator $I(z, \hx)$ blows up when the sampling point $z$ located on the hyperplanes $\hx\cdot(z-z_m)=0,\,m=1,2,\cdots, M$. Then we expect the superposition of $I_M(z, \hx)$ with respect to $\hx\in\Theta_L$ can be used to locate the locations $z_m, m=1,2,\cdots, M$. Based on this idea, we define
\be\label{IzM}
I_{M}(z):=\sum_{\hx\in\Theta_L}I_M(z, \hx).
\en
Numerically, the value of $I_M(z)$ is expect to be large if $z=z_m,\,m=1,2,\cdots, M$ and small otherwise.
After locating the locations $z_m,\,m=1,2,\cdots, M$, we may compute the scattering strength with the help of the following formula
\be\label{Itau}
\tau_m = \frac{1}{2K}\int_{0}^{K} \left\{u^{\infty}(\hx,k) e^{ik\hx\cdot z_m}+u^{\infty}(-\hx,k) e^{-ik\hx\cdot z_m}\right\}dk, \quad m=1,2,\cdots, M.
\en
Here, $\hx$ is chosen such that $\hx\cdot z_{m_1}\neq \hx\cdot z_{m_2}$ if $m_1\neq m_2$.
As mentioned in the proof of Theorem \ref{Uniqueness-nknM}, the right hand side extends to $\tau_m$ as $K\rightarrow \infty$. Thus by taking $K$ large enough, we hope to determine the scattering strength $\tau_m, \,m=1,2,\cdots, M$.

\subsection{Numerical methods for point like scatterers}

Now we turn to the scattering of plane wave by point like scatterers. The indicator \eqref{Iz1} with $u^{s}(x,k)$ replaced by $u^s(x,\hth, k)$ also works for locating a single point like scatterers. As in the previous subsection, we are more interested in the case with multiple point like scatterers.
We need some modifications due to the incident plane waves. Of practical interest is the backscattering inverse scattering problems, i.e., $\hx=-\hth$.
We assume we have the backscattering far field patterns $u^{\infty}(\hx,-\hx, k)$ for $\hx\in\Theta_L$ and $k\in (0, K)$. Define
\be\label{IzMx-obstacle}
\wi{I}_{M}(z, \hx):=\left|\int_{0}^{K} \left\{u^{\infty}(\hx,-\hx,k) e^{2ik\hx\cdot z}+u^{\infty}(-\hx,\hx,k) e^{-2ik\hx\cdot z}\right\}dk\right|, \quad z\in\R^n, \hx\in\Theta_L.
\en
Similarly, for large $K$, the indicator $\wi{I}_M(z, \hx)$ blows up when the sampling point $z$ located on the hyperplanes $\hx\cdot(z-z_m)=0,\,m=1,2,\cdots, M$. Then we further define
\be\label{IzM-obstacle}
\wi{I}_{M}(z):=\sum_{\hx\in\Theta_L}\wi{I}_M(z, \hx).
\en
Again, the value of $\wi{I}_M(z)$ is expect to be large if $z=z_m,\,m=1,2,\cdots, M$ and small otherwise.
After locating the locations $z_m,\,m=1,2,\cdots, M$, we may compute the scattering strength with the help of the following formula
\be\label{Itau-obstacle}
\tau_m = \frac{1}{2K}\int_{0}^{K} \left\{u^{\infty}(\hx,-\hx,k) e^{2ik\hx\cdot z_m}+u^{\infty}(-\hx,\hx,k) e^{-2ik\hx\cdot z_m}\right\}dk, \quad m=1,2,\cdots, M.
\en

\section{Numerical examples and discussions}
\label{NumExamples}
\setcounter{equation}{0}
This section is devoted to the some numerical examples to verify the effectiveness and robustness of the numerical methods proposed in the previous section. Note that both the theoretical results and numerical methods for the point obstacle case is quite similar to the point source case, therefore we only present the examples with point sources.
\subsection{One point source}
First, we consider the identification of one point source by the scattered fields at four sensors:
\ben
x_1=(2,0,0),\quad x_2=(0,2,0),\quad x_3=(0,0,2)\quad\mbox{and}\quad x_4=(-2,-2,-2).
\enn
The following three pairs of point sources with different locations and scattering strengths are considered.
\begin{itemize}
  \item $z_1=(1,1,1),\quad\tau_1=1+i;$
  \item $z_2=(1,0,1),\quad\tau_2=1-i;$
  \item $z_3=(0,1,1),\quad\tau_3=-1-i.$
\end{itemize}
We first use \eqref{Numx-z1} and \eqref{us-source} to get the strength, and then use
the indicator \eqref{Iz1} to capture the location of the point source. Table~\ref{indicator.one} gives the numerical results with different relative noise. In the numerical implementation,
 trapezoid integral formula is applied with $k_{\_}=1,k_+=100,dk=0.005$. Sampling space $0.05$ is used in the indicator \eqref{Iz1}.
\begin{table}
\begin{center}
\begin{tabular}{llllll }
\hline
True point&$z_1=(1,1,1)$ & $z_2=(1,0,1)$   &$z_3=(0,1,1)$   \\
& $\tau_1=1+i$ & $\tau_2=1-i$ &$ \tau_3=-1-i$\\
\hline
$0$ noise&$(1,1,1)$  & $(1,0,1)$ &$(0,1,1)$  \\
&$1.0024+ 0.9909i$ &$0.9931-1.0012i$ & $-0.9972-0.9947i$\\
\hline
$2\%$ noise&$(1,1,1)$ & $(0.95,-0.05,0.95)$ &$(0.05,1.05,1.05)$   \\
& $0.9750+1.0094i$&$0.9735-1.0055i$ & $ -1.0280- 0.9947i$\\
\hline
$5\%$ noise &$(1,1,1)\ $ & $(0.9,-0.15,0.9)$ & $(-0.1,0.95,0.9)$ \\
&$1.0252+0.8692i$&$0.9154-1.0180i$ &$-1.0787-0.8666i $\\
\hline
$10\%$ noise &$(1,1.05,1)\ $& $(0.95,0.15,1)$ &$(-0.1,0.85,0.8)$\\
&$0.9313+1.1136i$&$0.8008- 0.9240i$ &$ -0.8865-1.2529i$\\
\hline
\end{tabular}
\caption{One point source reconstruction using the scattered fields.}
\label{indicator.one}
\end{center}
\end{table}
We also test {\bf One point determination scheme} for the same points. Table~\ref{geo.one} give the results with different noise.\\
\begin{table}
\begin{center}
\begin{tabular}{llllll }
\hline
True point&$z_1=(1,1,1)$ & $z_2=(1,0,1)$   &$z_3=(0,1,1)$   \\
\hline
$0$ noise&$(1,1,1)$  & $(1,0,1)$ &$(0,1,1)$  \\
\hline
$2\%$ noise&$(0.99,1.02,0.99)$ & $(1.04,0.01,1.05)$ &$(-0.01,1.00,1.01)$   \\
\hline
$5\%$ noise &$(0.92,0.88,0.90)\ $ & $(1.11,0.14,1.03)$ & $(0.04,1.00,0.96)$ \\
\hline
$10\%$ noise &$(0.87,1.03,1.12)\ $& $(1.25,-0.11,1.09)$ &$(-0.03,1.21,1.20)$\\
\hline
\end{tabular}
\caption{Locating the point source by the {\bf One point determination scheme}.}
\label{geo.one}
\end{center}
\end{table}

Next, the construction using the far field pattern is tested, three observation directions $\hat{x}_1=(1,0,0),\hat{x}_2=(0,1,0),\hat{x}_3=(0,0,1)$ are chosen.
Equations \eqref{Numhxdotz1} and \eqref{uinf-source} are used to get the location and scattering strength respectively. The results are summarized in table~\ref{far.one}
with different noise.

\begin{table}
\begin{center}
\begin{tabular}{llllll }
\hline
True point&$z_1=(1,1,1)$ & $z_2=(1,0,1)$   &$z_3=(0,1,1)$   \\
& $\tau_1=1+i$ & $\tau_2=1-i$ &$ \tau_3=-1-i$\\
\hline
$0$ noise&$(1,1,1)$  & $(1,0,1)$ &$(0,1,1)$  \\
&$0.9974+1.0026i$ &$1.0026-0.9974i$ & $-1.0000-1.0000i$\\
\hline
$2\%$ noise&$(0.99,1.00,1.00)$ & $(1.01,0,0.99)$ &$(0,1,1)$   \\
& $1.0279+1.0006i$&$1.0251-0.9981i$ & $ -1.0032 -1.0032i$\\
\hline
$5\%$ noise &$(1.02,0.99,0.98)\ $ & $(0.98,0,1.00)$ & $(0,1.00,0.96)$ \\
&$1.0093+1.0480i$&$0.9753-1.0121i$ &$-0.9793-0.9793i $\\
\hline
$10\%$ noise &$(1.08,0.95,0.94)\ $& $(0.97,0,1.02)$ &$(0,0.98,1.04)$\\
&$0.8724+1.0142i$&$1.0434-1.1004i$ &$ -0.9350-0.9350i$\\
\hline
\end{tabular}
\caption{One point source reconstruction using the far fields.}
\label{far.one}
\end{center}
\end{table}

 \subsection{Multiple point sources}
In this part, more interesting case with multiple point sources are considered.
For simplicity, we present the examples in two dimensional case. The three dimensional case are similar by using the same indicator.
 We first use $\eqref{IzMx}$ and
$\eqref{IzM}$ to locate the positions of the point sources, then $\eqref{Itau}$ is used to get the strength. In the sequel, we take $k_{\_}=40,k_+=200, dk=1$ and $10\%$ relative noise in the numerical results.

In the first example, we consider five points sources located at
\ben
z_1=(1,0),\quad z_2=(0,1),\quad z_3=(-1,0),\quad z_4=(0,-1),\quad z_5=(0,0).
\enn
We use the $rand$ in MATLAB to set the true strengths.
Figure \ref{crossone} shows the reconstructions with one pair of directions. Clearly, the parallel lines containing the unknown points with normal in the observation direction are clearly reconstructed. To locate the points, we take $8$ observation directions $\hat x_j=(\cos\alpha_j,\sin\alpha_j)$ where $\alpha_j=2\pi j/8, j=0,\cdots,7$. Figure \ref{cross} gives the reconstruction of the five points. Table~\ref{fivetau} gives the comparison of true strength and computed strength.

In the final example, we use $35$ points to characterize the word "AMSS" (abbreviation for Academy of Mathematics and Systems Science). All the strengths are set to $1$. As shown in Figure~\ref{AMSS}, the $35$ points are well captured, even $10\%$ relative noise is considered. In particular, for this complicated example, $32$ equally distributed observation directions are used, where the direction number is smaller than the point number.

\begin{figure}[htbp]
  \centering
  \subfigure[\textbf{$\hat x=\pm(0,1)$.}]{
    \includegraphics[height=1.1in,width=1.3in]{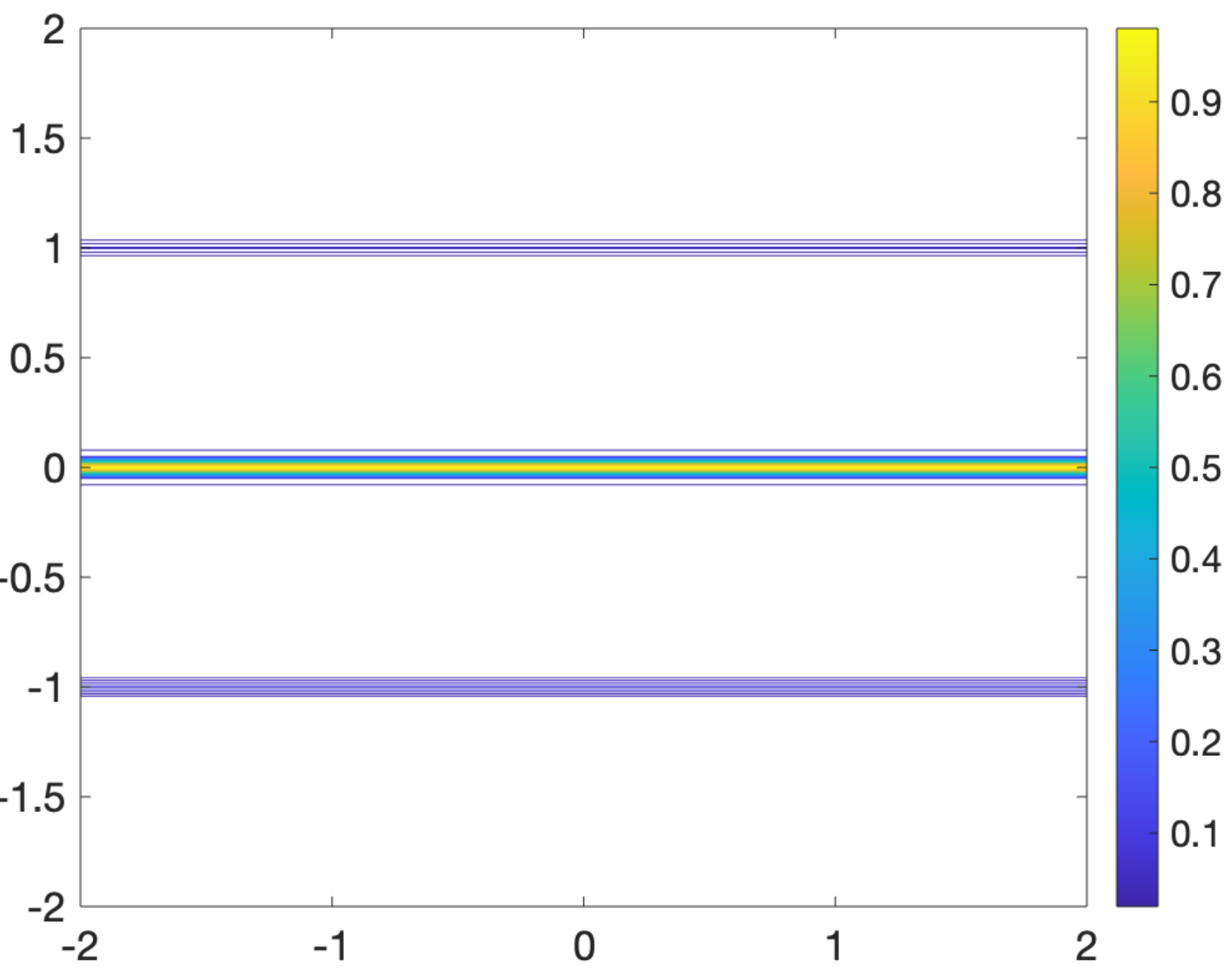}}
  \subfigure[\textbf{$\hat x=\pm(\frac{\sqrt{2}}{2},-\frac{\sqrt{2}}{2})$.}]{
    \includegraphics[height=1.1in,width=1.3in]{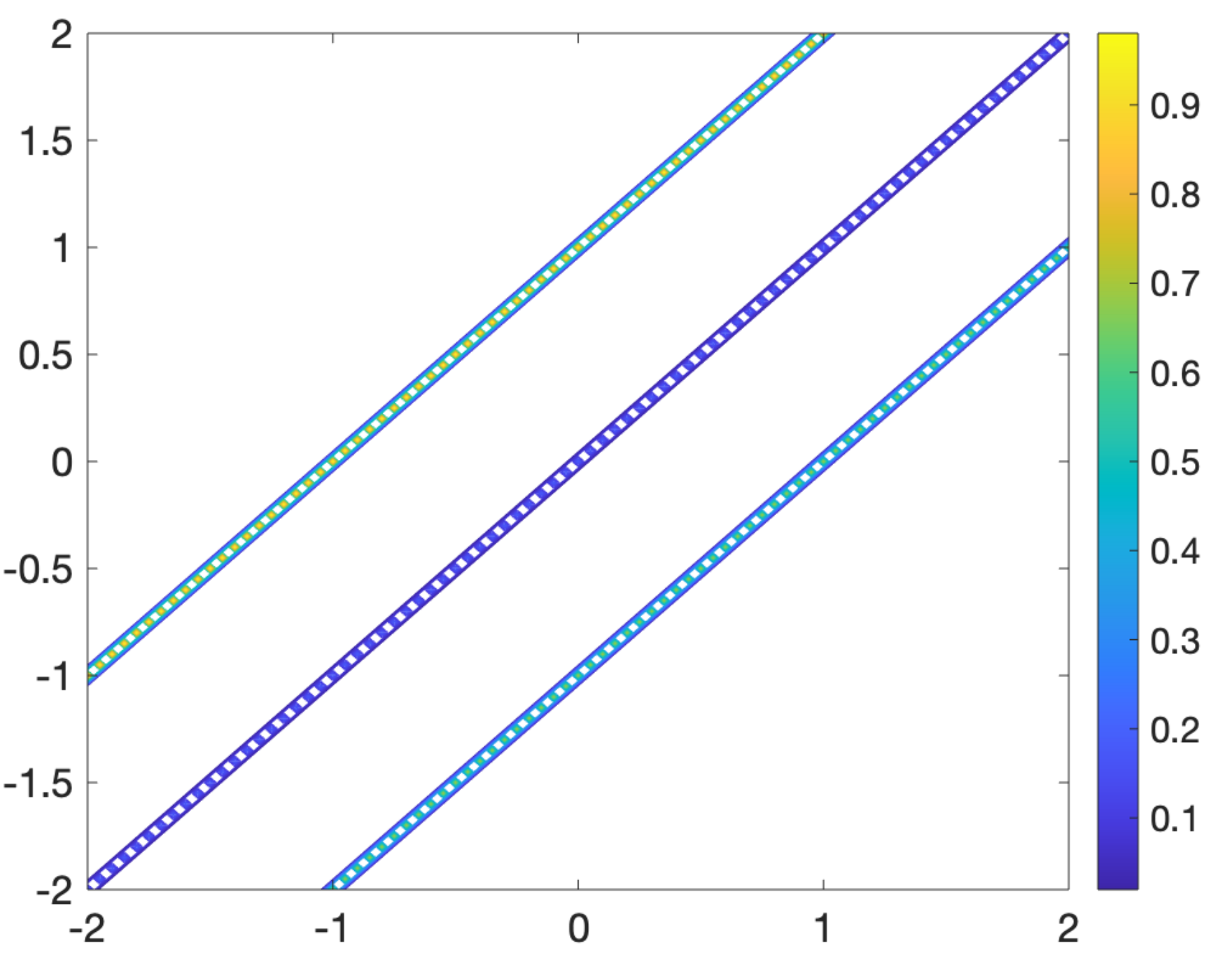}}
  \subfigure[\textbf{$\hat x=\pm(1,0)$.}]{
    \includegraphics[height=1.1in,width=1.3in]{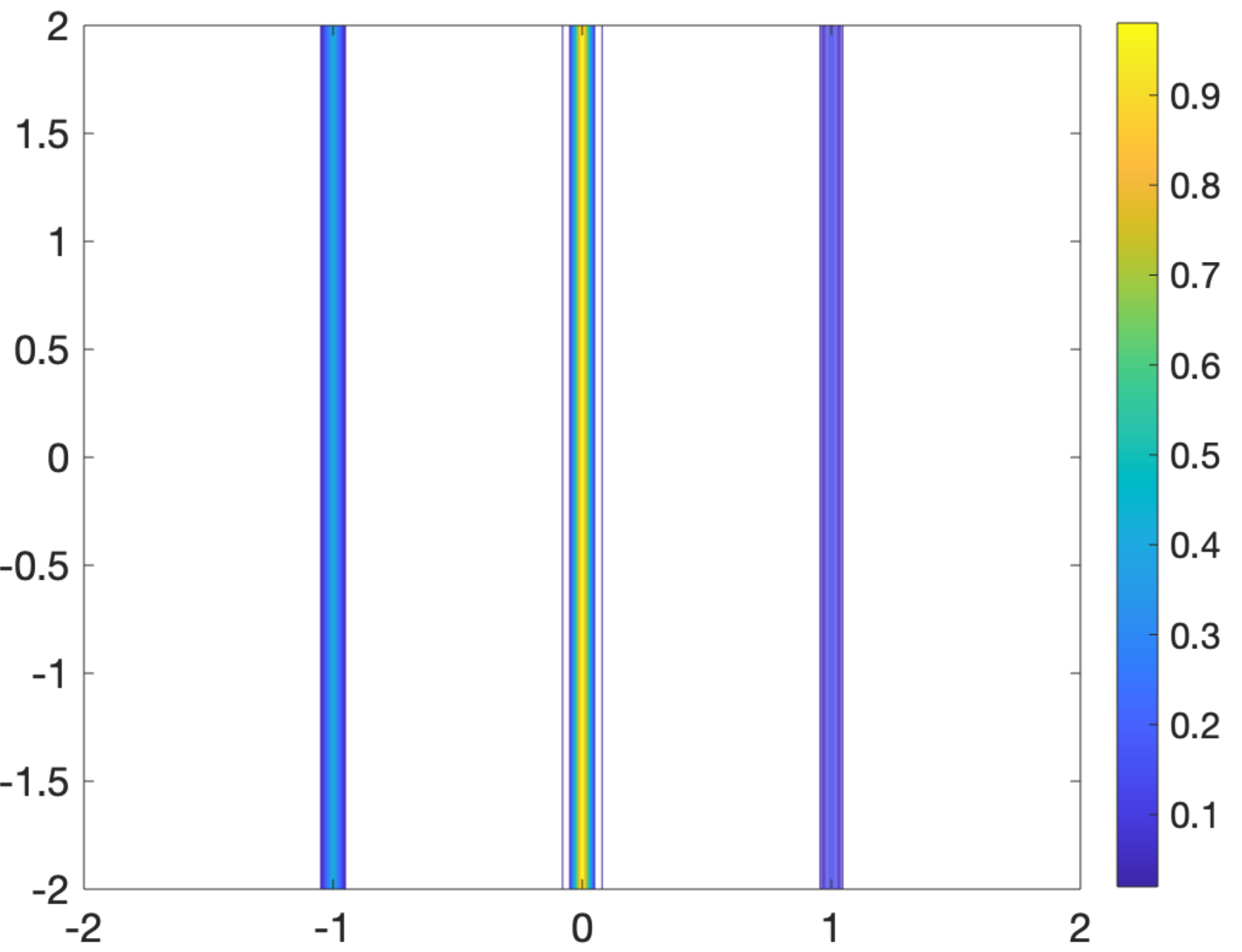}}
  \subfigure[\textbf{$\hat  x=\pm(\frac{\sqrt{2}}{2},\frac{\sqrt{2}}{2})$.}]{
    \includegraphics[height=1.1in,width=1.3in]{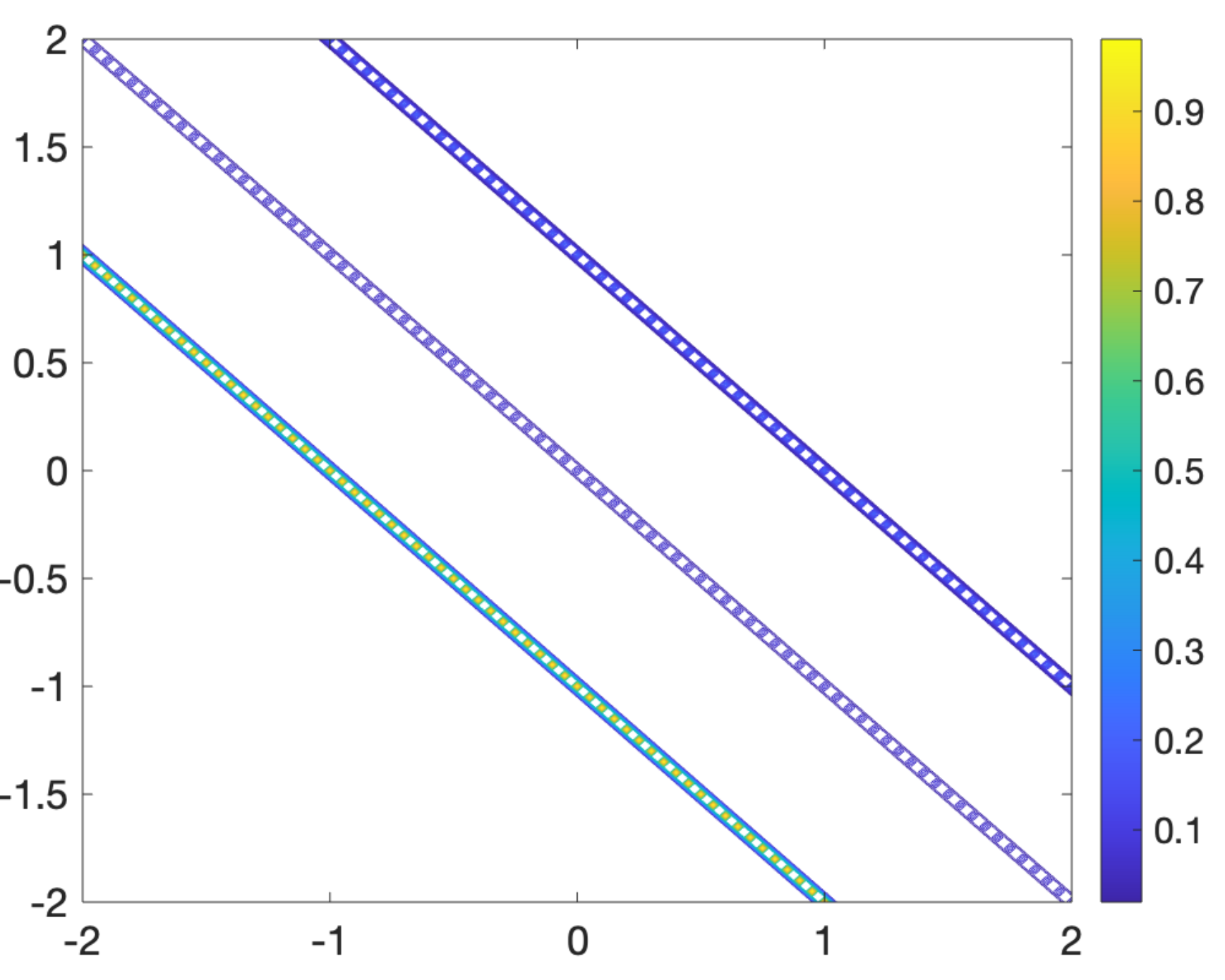}}
\caption{Reconstruction by one pair of directions  with  $10\%$ noise.}
\label{crossone}
\end{figure}

\begin{figure}[htbp]
  \centering
    \includegraphics[height=1.6in,width=1.8in]{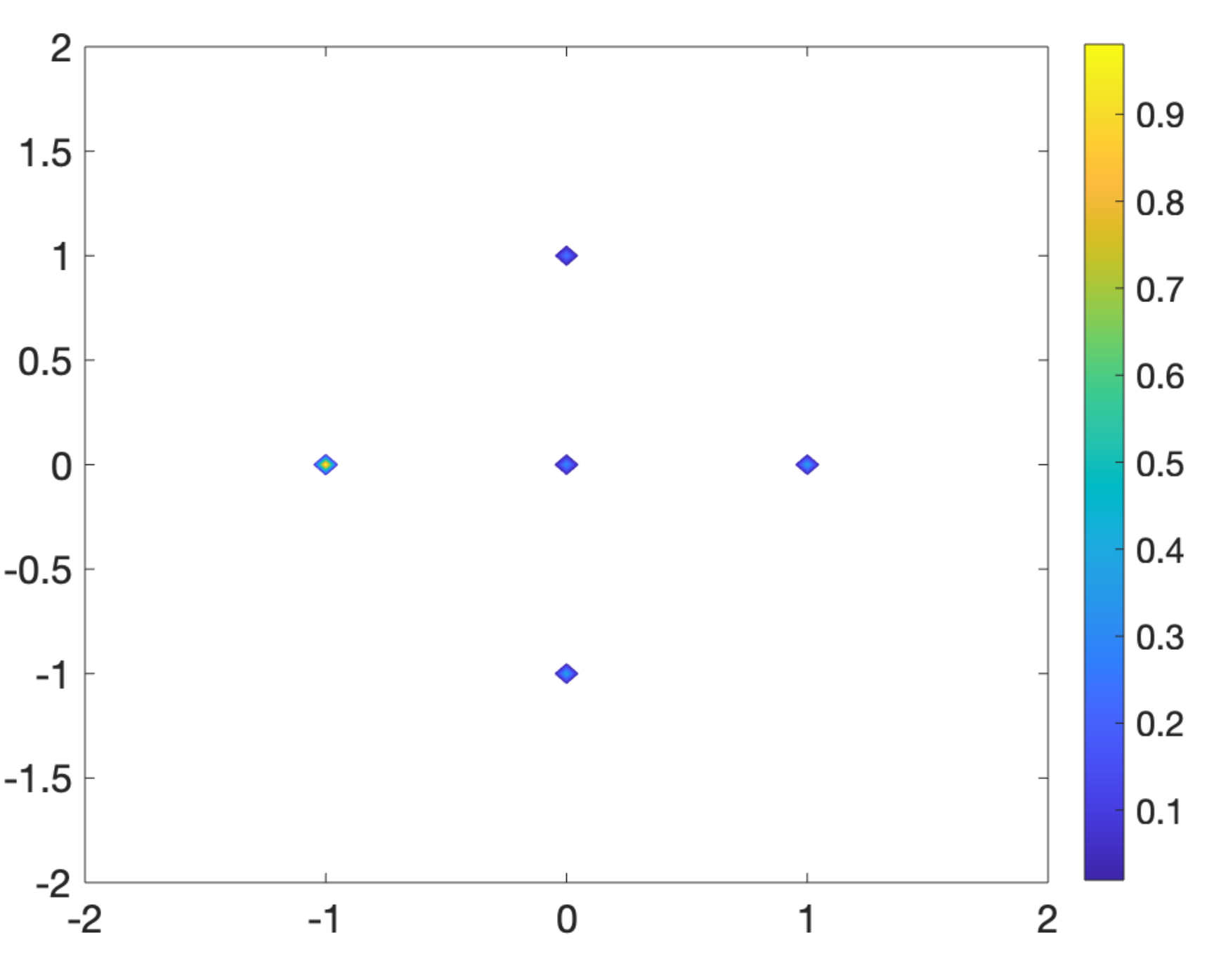}
\caption{Locating the five point sources with 8 directions and $10\%$ noise.}
\label{cross}
\end{figure}

\begin{table}
\begin{center}
\begin{tabular}{llllll }
\hline
&True strength&\qquad Computed strength      \\
\hline
$\tau_1$ & $0.119437 + 0.858134i$ &\qquad $ 0.117949 + 0.860545i $\\
\hline
$\tau_2$ & $0.931100 + 0.056194i$ &\qquad $ 0.939913 + 0.061328i $\\
\hline
$\tau_3$ & $0.994541 + 0.975031i$ &\qquad $ 1.001599 + 0.981825i $\\
\hline
$\tau_4$ & $0.406819 + 0.595928i$ &\qquad $ 0.405445 + 0.607381i$\\
\hline
$\tau_5$ & $0.117482 + 0.901291i$ &\qquad $ 0.131289 + 0.911853i$\\
\hline
\end{tabular}
\caption{Reconstruction of scattering strengths using the far fields with  $10\%$ noise. We take $\hx=(\cos(\pi/16),\sin(\pi/16))$ in this example.}
\label{fivetau}
\end{center}
\end{table}

\begin{figure}[htbp]
  \centering
    \includegraphics[height=1.6in,width=5.6in]{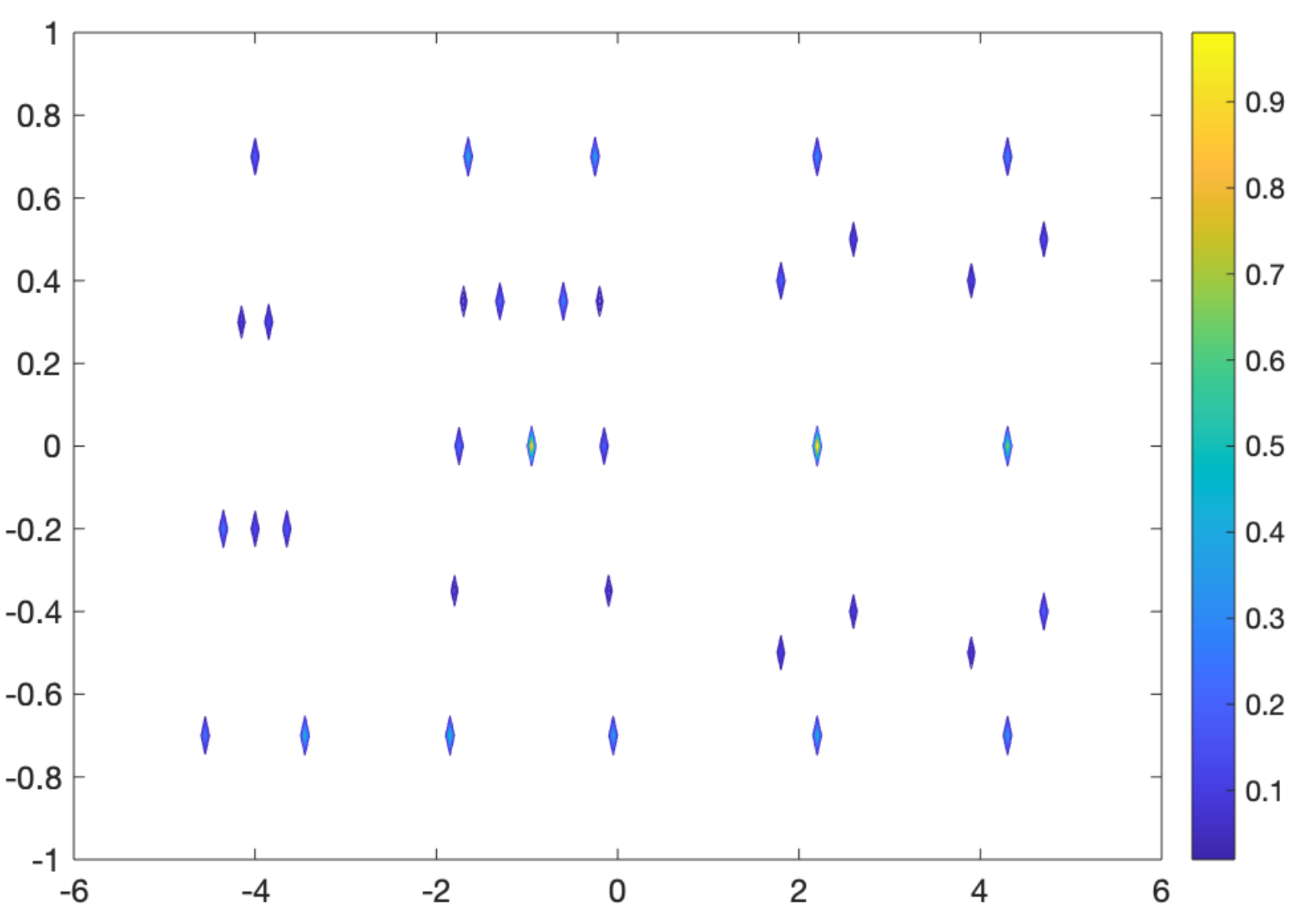}
\caption{Reconstruction of "AMSS" with $32$ directions  and $10\%$ noise.}
\label{AMSS}
\end{figure}

\section{Conclusions and remarks}

We consider the scattering by point objects using measurements at finitely many sensors. Both the uniqueness and numerical methods for identifying the point objects have been studied. The numerical examples further verify the effectiveness and robustness of the proposed numerical methods.

Finally, we want to remark that our numerical algorithms are also works for small inclusions. This is due to the fact that the point objects can be regarded as an approximation of small inclusions.

\section*{Acknowledgement}
The research of X. Ji is supported by the NNSF of China under grants 91630313 and 11971468,
and National Centre for Mathematics and Interdisciplinary Sciences, CAS.
The research of X. Liu is supported by the NNSF of China under grant 11971701, and the Youth Innovation Promotion Association, CAS.

\bibliographystyle{SIAM}

\end{document}